\documentclass{amsart}
\usepackage[utf8]{inputenc}

\usepackage{amsfonts, amssymb, amsmath, amsthm, latexsym}
\usepackage{thmtools} % I dont know why we need thmtools but it solved the problem of cleverref naming in overleaf. Load, in this order: amsthm, thmtools, cleverref. according to https://tex.stackexchange.com/questions/19104/cleveref-with-counters-with-same-name
\usepackage{mathtools}

\usepackage{hyperref}
\hypersetup{
	colorlinks=true,
	linkcolor=blue,
	filecolor=magenta,      
	urlcolor=blue,
	pdfpagemode=FullScreen,
}
\usepackage[capitalise]{cleveref}
\usepackage{comment}
\usepackage{xfp}
\usepackage{svg} 
\usepackage{tikz, tikz-cd}
\usetikzlibrary{positioning,arrows,arrows.meta, decorations.pathmorphing,backgrounds,patterns,decorations.markings,decorations.pathreplacing,decorations.text,shapes.geometric,svg.path}
\tikzset{negated/.style={
        decoration={markings,
            mark= at position 0.5 with {
                \node[transform shape] (tempnode) {$\backslash$};
            }
        },
        postaction={decorate}
    }
}

\usepackage{relsize}
\tikzset{fontscale/.style = {font=\relsize{#1}}	}
\usepackage[shortlabels]{enumitem}
\usepackage{caption, subcaption}
\usepackage{blindtext}
\usepackage[export]{adjustbox}
\graphicspath{ {img/} }

\newtheorem{introtheorem}{Theorem}[]

\newtheorem{theorem}{Theorem}[section]
\newtheorem{lemma}[theorem]{Lemma}
\newtheorem{proposition}[theorem]{Proposition}
\newtheorem{corollary}[theorem]{Corollary}
\theoremstyle{definition}

\newtheorem{definition}[theorem]{Definition}
\newtheorem{remark}[theorem]{Remark}
\newtheorem{notation}[theorem]{Notation}

\newcommand{\mcomment}[1]{}%\marginpar{\begin{center}\scriptsize{#1}\end{center}}} % Comentario al margen
\newcommand{\lam}[1]{\textcolor{olive}{}}%\textbf{LAM:}}}% #1}}
\newcommand{\pll}[1]{\textcolor{purple}{}}%\textbf{PLL:} #1}}

\newcommand{\mycomment}[1]{{\color{olive}}}%#1}}
\newcommand{\paloma}[1]{{\color{purple}}}%#1}}
\usepackage[normalem]{ulem}
\newcommand\so{\bgroup\markoverwith	{\textcolor{olive}{\rule[.5ex]{2pt}{0.4pt}}}\ULon} %strikeout

\newcommand*{\A}{\mathcal{A}}

\newcommand{\Z}{\mathbb{Z}}

\newcommand{\slg}{\textnormal{SL}_G}
\newcommand{\fftp}{\textnormal{FFTP}}
\newcommand{\slfftp}{\textnormal{SL-FFTP}}
\newcommand{\sll}{<_{\textnormal{SL}}}

\newcommand{\csl}{\textnormal{C}_{\textnormal{SL}}}
\newcommand{\slrep}{\textnormal{SL}}
\newcommand{\lex}{<_{\textnormal{lex}}}

\title{Groups with finitely many shortlex cones}
\author[L. Asencio-Mart\'{i}n]{Luc\'{i}a Asencio-Mart\'in}
\address{School of Mathematics, Statistics and Physics, Newcastle University,
Newcastle upon Tyne
NE1 7RU, United Kingdom}
\email{L.Asencio-Martin2@newcastle.ac.uk}

\author[P. López-Larios]{Paloma López-Larios}
\address{Facultad de Ciencias Matemáticas, Universidad Complutense de Madrid,
Madrid
28040, España}
\email{pallop05@ucm.es}
\begin{document}
\begin{abstract}
    We show that a finitely generated group $G=\langle\Sigma\rangle$ satisfying the falsification by fellow traveller property also satisfies the shortlex falsification by fellow traveller property. 
    This implies that there are finitely many shortlex cone types and, therefore, that the corresponding language of shortlex representatives $\slg\subset (\Sigma\cup\Sigma^{-1})^*$ is a regular language. Following the example of M. Elder, we prove that the converse of the previous statements do not hold.
\end{abstract}
\maketitle
\section{Introduction}

 The \emph{falsification by fellow traveller property} (FFTP) was introduced in \cite{Neumann_Shapiro_1995} by Neumann and Shapiro, who were inspired by the ideas of Cannon in \cite{cannon84}. Informally, a graph has the FFTP if every non-geodesic path in the graph admits a shorter path, with the same endpoints, that remains close to the first one. We say that a group $G$ with finite generating set $\Sigma$ has the FFTP with respect to $\Sigma$ if the Cayley graph Cay(G, $\Sigma$) has it. 

 From the work of Neumann and Shapiro, one can see that if a group G has the FFTP with respect to some generating set $\Sigma$, then the language of geodesics over this generating set is regular. In addition to the results in the grounding works of Cannon, Neumann and Shapiro, the falsification by fellow traveller property has been the topic of many research articles. Indeed, it has been shown that groups with FFTP are of type $F_3$ and their Dehn functions are at most quadratic  \cite{Elder2002}, that they have rational growth, and that having a regular set of geodesics does not imply having FFTP \cite{elder05}. Moreover, there are many works providing new examples of families of groups that satisfy FFTP: Dyer groups \cite{howarth26}, Coxeter groups \cite{noskov01}, dihedral Artin groups \cite{cc25}, Garside groups \cite{holt10}, Artin groups of large type \cite{hr12} and relatively hyperbolic groups \cite{ac16}.

For a fixed total order in the letters of an alphabet $A$, it is possible to define a well-ordering on the set of words over the alphabet $A$, $A^*$, which is called the \emph{shortlex ordering} of $A^*$ (see \cref{def:shortlex_ordering}). If $G$ is a group with generating set $\Sigma$, we can consider the shortlex ordering on $(\Sigma \cup \Sigma^{-1})^*$ and take the unique minimal representative, called shorlex representative, of each element of the group with respect to this ordering.
We say that a group $G$ together with a finite generating set $\Sigma$ satisfies the \emph{shortlex falsification by fellow traveller property} (SL-FFTP) if there exists a global constant $K$ such that, for every word that is not a  shortlex representative, we can find another representative of the same group element which is smaller in the shortlex ordering and $K$-fellow travels with it. 

The main goal of this paper is to prove the following theorem:
\begin{introtheorem}[\cref{prop: fftp implies regular language}]
    Let $G$ be a group with finite generating set $\Sigma$ satisfying $\slfftp$, then the set of shortlex representatives $\slg$ is a regular language.
\end{introtheorem}
The strategy to prove this result is to mimic the classical argument by Neumann and Shapiro for showing that having FFTP implies that the set of geodesics is regular. Essentially, we will define a shortlex analogue of the cone of an element (see \cref{def: slcone}) and conclude that, under $\slfftp$, there are finitely many of them, which allows us to build a finite state automaton recognising the language of minimal shortlex representatives of elements of the group.

In the classical case of $\fftp$, one can use the fact that the language of geodesics is regular to prove that the group has rational growth \cite{Neumann_Shapiro_1995}; however, the Heisenberg group gives a counterexample to the converse statement: this group has rational growth \cite{duchinshapiro19} but does not have a regular set of geodesics \cite{warshall10}. Analogously, it is true that a regular set of shortlex representatives implies that the group has rational growth (\cref{coro:slfftp_rational_growth}) and the Heisenberg group, following the arguments in \cite{warshall10}, can be seen to have rational growth but not a regular set of shortlex representatives with respect to the standard generators.

We have already mentioned that groups with $\fftp$ have a regular set of geodesics. In \cite{elder05}, the author shows that the converse is not true: the group  $G
$ with presentation $\langle a, t \;|\; t^2=1, atat =tata\rangle$ does not have $\fftp$ but the language of geodesics is regular. We characterise for which total orders on $\{a,a^{-1}, t\}$ the group has $\slfftp$ and we show:

\begin{introtheorem}[\cref{prop:slg_regular_not_slfftp,prop:ordering_non_slfftp,prop:some_orderings_slfftp}]
    There exists a group $G$ and a generating set $\Sigma$ such that $(G,\Sigma)$ does not have $\fftp$, it has $\slfftp$ for some total orderings on $\Sigma \cup \Sigma^{-1}$ but not for all of them, and the language of shortlex representatives $\slg$ is a regular language for any total ordering. 
\end{introtheorem}
The following diagram sums up the relationship and the parallellisms between the classical falsification by fellow traveller property and the shortlex version that we introduce:

\begin{centering}\[
\begin{tikzcd}[arrows=Rightarrow, column sep=0.8cm, row sep=0.1cm, every arrow/.append style={shift left=0.8ex}]
    G \text{ has FFTP }\hspace{0.3cm}
    \arrow[]{dd} 
    \arrow{r}
    &  \text{regular geodesics}
    \arrow[negated]{l} 
    \arrow[]{rd} & 
    \\
    & & \hspace{0.5cm}G\text{ has rational growth}
    \arrow[negated]{lu}
    \arrow[negated]{ld}
    \\    
    G \text{ has SL-FFTP}
    \arrow[negated]{uu} 
    \arrow{r}
    &  \hspace{0.4cm}\text{regular } \slg\hspace{0.4cm}
    \arrow[negated]{l} 
    \arrow[]{ru} &   
\end{tikzcd}\]
\end{centering}

The rest of this document is structured as follows: \cref{sec:defs} provides preeliminary definitions regarding regularity of languages and the shortlex version of the cone of an element. \cref{sec:slfftp} contains the definitions of the falsification by fellow traveller property and its shortlex version, as well as the proof that a group with SL-FFTP has a regular set of shortlex representatives and, as a consequence, it has rational growth; we discuss why the Heisenberg group provides a counterexample to the converse. \cref{sec:fftp_vs_slfftp} contains further results on the relation between shortlex FFTP and the regularity of shortlex representatives.

\section{Shortlex Cones}\label{sec:defs}
\begin{definition}[Finite state automaton]\label{def:automaton}
A  \emph{finite state automaton} (FSA)  is a tuple $\A= (S, V, E, V_F, v_0)$ 
 where $S$ is a finite set of letters (the \emph{alphabet}),  $V$ is the set of vertices or 
\emph{states}, 
$E$ is the set of $S$-labelled edges or \emph{transitions} of the form $(v, s, w)\in V\times \Sigma\times V$, 
$v_0\in V$ is the \emph{initial state} and $V_F\subset V$ is the set of 
\emph{accepting states} of the automaton.
\end{definition}
We say that $(v, s, w)\in E$ is 
an \emph{$s$-labelled} transition.
\begin{definition}[Language accepted by FSA, regular language]
To each finite state automaton $\A$ we can assign the set $L(\A)\subset S^*$ 
of words \emph{accepted} by the automaton, which are the words given by the concatenation of edge labels of paths in $\A$ that start in $v_0$ and end in one of the 
accepting states;  we call $L(\A)$ the \emph{language} accepted by $\A$.
 We say that a  subset $L$ of $S^*$ is \emph{regular} if it is the language $L(\A_L)$  accepted by a finite state automaton $\A_L$.
\end{definition}
For further definitions related to finite state automata and regular languages we refer the reader to \cite[\S1.1 and \S1.2]{epstein92}.

Next, we introduce the shortlex ordering for words over an alphabet with respect to a given total ordering of the alphabet.
\begin{definition}[Shortlex ordering]\label{def:shortlex_ordering}
    Let $(S, <)$ be a totally ordered alphabet. We will consider two different order relations in $S^*$:
    \begin{itemize}
        \item The \emph{lexicographical order}: $w_1,w_2\in S^*$, $w_1\lex w_2$ if and only if $w_1$ is a strict prefix of $w_2$ or the letter in the first position in which $w_1$ and $w_2$ differ is smaller for $w_1$.
        \item The \emph{shortlex order}: for $w_1,w_2\in S^*$, $w_1\sll w_2$ if and only if $w_1$ is shorter than $w_2$ or they have the same length and $w_1\lex w_2$.
    \end{itemize}  
\end{definition}

\begin{notation}\label{notation:sl}
    The shortlex order is a well-ordering. In particular, if $G$ is a group generated by $\Sigma$ and we consider the shortlex order on $(\Sigma\cup \Sigma^{-1})^*$ with respect to some ordering of $\Sigma\cup \Sigma^{-1}$, there is a well-defined minimal representative for each element $g$ of the group $G$ under this order, which will be denoted $\slrep_{(\Sigma\cup\Sigma^{-1}, <)}(g)$ or simply $\slrep(g)$ if $(\Sigma\cup\Sigma^{-1}, <)$ is understood by the context. We will denote 
    by $\slrep_{(G, \Sigma\cup\Sigma^{-1}, <)}$ the language of these minimal representatives of elements of $G$, and write $\slg$ for short if $(G, \Sigma\cup\Sigma^{-1}, <)$ is understood.
\end{notation}

\begin{definition} \label{def: slcone}
    Let $G$ be a group with finite generating set $\Sigma$ and let $g$ be an element of $G$. The $\slrep$-cone of $g$ is defined as \[\csl(g):=\{w\in (\Sigma\cup \Sigma^{-1})^*: \slrep(g)w\in \slg \}.\]
\end{definition}

\begin{lemma}  \label{lemma: construction automaton}
    A group $G$ with finite generating set $\Sigma$  has finitely many $\slrep$-cones if and only if
     $\slg$ is a regular language.
\end{lemma}
\begin{proof}
    If $\slg$ is a regular language, let $\A= (\Sigma\cup\Sigma^{-1}, V, E, V_F, v_0)$ by a finite state automaton with $L(\A)=\slg$. We claim that there is an injective map from the set of $\slrep$-cones of $G$ to the finitely many states of $\A$. To see this, let $g$ be an element in $G$ and let $w=\slrep(g)$. The word $w$ is the label of an accepted path in $\A$ from its start state $v_0$ to an accepting state $v_g\in V_F$: the injection is defined by $\csl(g)\mapsto v_g$. Indeed, let $g$, $h$ be two elements of $G$ with $v_h=v_g$. By definition of the automaton $\A$, $\csl(g)$ and $\csl(h)$ coincide because they are equal to the set of labels of paths from $v_g=v_h$ to any accepting state of $\A$.
    
    For the converse, a finite automaton recognizing $\slg$ can be built as follows. Take the finite set of $\slrep$-cones of $G$ to be the states of the automaton, all states are final and there is a unique initial state: the $\slrep$-cone of $1_G$. For each $s\in \Sigma\cup \Sigma^{-1}$ and $g\in G$, there is a transition from $\csl(g)$ to $\csl(gs)$ if and only if $s\in \csl(g)$.
\end{proof}

\section{FFTP and SL-FFTP}\label{sec:slfftp}

Given a graph $X$, a  \emph{(combinatorial) path} $p$ is a sequence $v_0e_1v_1\ldots e_nv_n$ where $v_0,\ldots,v_n$ are vertices in $X$ and, for $i = 1,\ldots, n$, $e_i$ is an edge between $v_i$ and $v_{i+1}$. The \emph{length} of $p$ is the number $n$, and the vertices $v_0$, $v_n$ are called the \emph{endpoints} of the path. The path $p$ is \emph{geodesic} if its length is minimal among the lengths of all the paths in $X$ with endpoints $v_0$,$v_n$.

\begin{definition}[Fellow traveller for graphs]\label{def:fellow_traveller_graphs}
   For $K\geq 0$, we say that two paths $p = v_1\ldots v_n$, $q=u_1\ldots u_m$ in a graph $X$ \emph{asynchronously $K$-fellow travel} if there are non-decreasing maps $f\colon [0, n]\to[0, m]$, $g\colon [0, m]\to[0, n]$ satisfying that the vertices $v_i, u_{f(i)}$, as well as the vertices $u_j, v_{g(j)}$ are at distance at most $K$ in the graph for all $i\in[0, n]$ and $j\in[0, m]$ respectively.
\end{definition}
\begin{definition}[FFTP for graphs]\label{def:fftp_graphs}
    A graph $X$ has the \emph{ falsification by fellow traveller property} if there is a constant $K$ such that for every non-geodesic path $p$ there exists a path $q$ with the same endpoints as $p$, that is shorter than $p$ and that asynchronously $K$-fellow travels with it.
\end{definition}
These concepts are translated to a group using its Cayley graph with respect to a given set of generators. For any group $G$ with finite generating set $\Sigma$, its elements are represented by words $w\in(\Sigma\cup \Sigma^{-1})^*$. Each such word $w$ corresponds to a combinatorial path in the Cayley graph of $G$ with respect to $\Sigma$ starting at a fixed basepoint, take the basepoint to be the vertex corresponding to the identity $1_G$. 

The length of such a word $w$, which coincides with the length of the corresponding path, will be denoted by $|w|$. The length of the shortest $v\in (\Sigma\cup \Sigma^{-1})^*$ such that $v=_Gw$ is denoted by $|w|_G$. We say $w$ is geodesic if its corresponding path is geodesic or, equivalently, if $|w|=|w|_G$. 

\begin{definition}[FFTP for groups] \label{def: FFTP}
	A group $G$ with generating set $\Sigma$ is said to have the \emph{falsification by fellow traveller property} (FFTP) \emph{with respect to $\Sigma$} if the Cayley graph of $G$ with respect to $\Sigma$ has the asynchronous falsification by fellow traveller property.
\end{definition}

In this paper we propose a variation of $\fftp$ which, as far as we know,  is not explicit in the literature.
\begin{definition}[SL-FFTP] \label{def: SLFFTP}
	Let $G$ be a group with generating set $\Sigma$, and fix $<$ a total ordering on $\Sigma\cup\Sigma^{-1}$. The group $G$ is said to have the \emph{shortlex falsification by fellow traveller property} ($\slfftp$) \emph{with respect to $\Sigma$} if there exists a constant $K\geq 0$ such that, for any $w\not \in \slg$, there exists a word $v \in (\Sigma\cup \Sigma^{-1})^*$ that asynchronous $K$-fellow-travels with $w$ and satisfies that $w=_Gv$ and  $v\sll w$.
\end{definition}

From now on, $G$ will always be a group with finite generating set $\Sigma$. We fix a total order in $\Sigma\cup \Sigma^{-1}$. Every time we use the shortlex ordering in $(\Sigma\cup \Sigma^{-1})^*$, we will assume that we are referring to the shortlex order induced by this fixed total order.

We first observe the following implication.

\begin{lemma} \label{lem: fftp implies slfftp}
    If $G$ satisfies $\fftp$ with respect to $\Sigma$, then $G$ satisfies $\slfftp$ with respect to $\Sigma$.
\end{lemma}

\begin{proof}
Let $K$ be the $\fftp$ constant of $G$ with respect to $\Sigma$. Let $w\not \in \slg$. If $w$ is not geodesic, there exists a shorter word $v\in (\Sigma\cup \Sigma^{-1})^*$ that $K$-fellow travels with $w$ and represents the same element of $G$, since $G$ has  $\fftp$. In particular, $v\sll w$, so $\slfftp$ is satisfied in this case. Suppose now that $w$ is geodesic, since $w\not \in \slg$, then there exists $u\in (\Sigma\cup \Sigma^{-1})^*$ with $|u|=|w|$ admitting a factorization as the following: $u=u_1xu_2$ and  $w=w_1yw_2$, with $u_i,w_i\in (\Sigma\cup \Sigma^{-1})^*$, $x,y\in \Sigma\cup \Sigma^{-1}$, $u_1=w_1$ and $x<y$. The word $x^{-1}yw_2$ is not geodesic (since it represents the same element as $u_2$ and it has length $|u_2|+2$) so, by $\fftp$, there exists a word $v$ that $K$-fellow travels with $x^{-1}yw_2$, represents the same element of $G$ and satisfies $|u_2|\leq |v|<|x^{-1}yw_2|=|u_2|+2$. In particular, $|v|=|u_2|$ or $|v|=|u_2|+1$. By $\fftp$, we may assume that $|v|=|u_2|$. Then, $w_1xv$ $(2K)$-fellow travels with $w=w_1yw_2$, $w_1xv\sll w$ and $w=_Gw_1xv$, so $G$ has $\slfftp$.
\end{proof}

We show the analogue
to the fact that $\fftp$ implies finitely many cone types.

\begin{proposition} \label{lem: fftp implies finite number of sl cones}
    Let $G$ be a group with finite generating set $\Sigma$. If $G$ satisfies the $\slfftp$ with respect to $\Sigma$, then $G$ has a finite number of $\slrep$-cones.
\end{proposition}

\begin{proof}
    Let $K$ be the SL-FFTP constant. Denote by $B_K(1_G)$ the ball of radius $K$ and center $1_G$ in the Cayley graph of $G$ with respect to $\Sigma$. For $g\in G$ and $h\in B_K(1_G)$, we define
    \[\Delta_g(h):= |gh|_G-|g|_G,\]

\[
L_g(h):=
\left\{
\begin{array}{ll}
0,& \text{if $\slrep(gh)\lex\slrep(g)$},\\
1, & \text{if not},
\end{array}
\right.
\]
    
    \[
P_g(h):=
\left\{
\begin{array}{ll}
0,& \text{if $\slrep(gh)$ is not a prefix of $\slrep(g)$},\\
1, & \text{if $\slrep(gh)$ is a prefix of $\slrep(g)$}.
\end{array}
\right.
\]

Notice that $\Delta_g(h)\in [-K,K]$ for all $g\in G, h \in B_K(1_G)$. Thus,  we may define:
\[\begin{array}{rcl}
    \mathcal{C}_g: B_K(1_G) & \longrightarrow & [-K,K]\times \{0,1\}^2   \\
     h&  \longmapsto & (\Delta_g(h), L_g(h), P_g(h)).
\end{array}\]

Observe that $\{\mathcal{C}_g\}_{g\in G}$ is a finite set. So, if we take $g,g'\in G$ with $\mathcal{C}_g(h)=\mathcal{C}_{g'}(h)$ for all $h\in B_K(1_G)$, it is enough to show that $\csl(g)=\csl(g')$ in order to prove that the number of $\slrep$-cones of $G$ is finite.

Assume that $g,g'\in G $ satisfy that $\mathcal{C}_g(h)=\mathcal{C}_{g'}(h)$ for all $h\in B_K(1_G)$. Suppose, for a contradiction, that there exists $w\in \csl (g)-\csl(g')$. Since $G$ satisfies $\slfftp$ and $\slrep(g')w$ is not in $\slg$, there is a word $v$ that $K$-fellow travels with $\slrep(g')w$ and such that $v \sll \slrep(g')w$ and $v=_G \slrep (g')w$. We distinguish five cases, the first one is analogous to the classical proof for FFTP and the last four are depicted in \cref{fig:cones}.
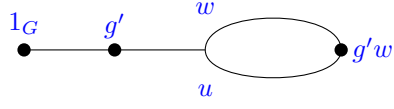
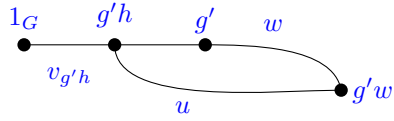
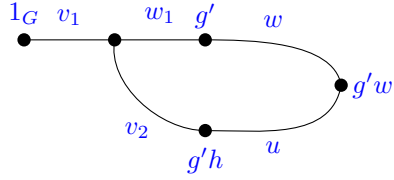
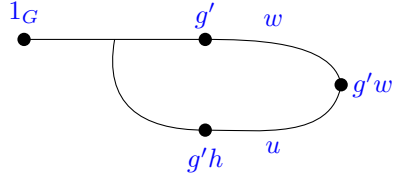
\begin{figure}[h]
    \centering
    \par\bigskip
    \begin{subfigure}[t]{\textwidth}
    \centering
    \begin{tikzpicture}[thick,scale=.6,arrowmark/.style 2 args={decoration={markings,mark=at position #1 with \arrow{#2}}}]%
      \tikzstyle{every node}=[circle, draw, fill=blue, color=blue,
                            inner sep=0pt, minimum width=6pt]
     \begin{pgfonlayer}{background}                       
                            \draw[-,>=stealth] (0,0) -- (2,0) -- (4, 0);
                            \draw[-] (4, 0) to[out=90, in=90, looseness=0.8] (7,0);
                            \draw[-] (4,0) to[out=-90, in=-90, looseness=0.7]  (7, 0);
    \end{pgfonlayer}                      
    \draw (0,0)  node[draw=none,fill=black, color=black, inner sep=0.5pt, minimum width=5pt, label = $1_G$] {};
    \draw (2,0)  node [draw=none,fill=black, color=black, inner sep=0.5pt, minimum width=5pt, label=$g'$]{};
    \draw (7,0)  node [draw=none,fill=black, color=black, inner sep=0.5pt, minimum width=5pt, label = {0:$g'w$}](gw){};
    \draw (4,0.5)  node [draw=none,fill=none, color=black, inner sep=0.5pt, minimum width=5pt, label = $w$](w){};
    \draw (4,-0.5)  node [draw=none,fill=none, color=black, inner sep=0.5pt, minimum width=5pt, label ={-90:$u$}](u){};
    \end{tikzpicture}
    \caption{Case 2.1.}
    \label{fig:21}
    \end{subfigure}%
    \vspace{1.5em}
    \begin{subfigure}[t]{\textwidth}
    \centering
    \begin{tikzpicture}[thick,scale=.6,arrowmark/.style 2 args={decoration={markings,mark=at position #1 with \arrow{#2}}}]%
      \tikzstyle{every node}=[circle, draw, fill=blue, color=blue,
                            inner sep=0pt, minimum width=6pt]
     \begin{pgfonlayer}{background}                       
                            \draw[-,>=stealth] (0,0) -- (2,0) -- (4, 0);
                            \draw[-] (4, 0) to[out=00, in=100, looseness=0.8] (7,-1);
                            \draw[-] (2,0) to[out=-90, in=180, looseness=0.7]  (7, -1);
    \end{pgfonlayer}                      
    \draw (0,0)  node[draw=none,fill=black, color=black, inner sep=0.5pt, minimum width=5pt, label = $1_G$] {};
    \draw (4,0)  node [draw=none,fill=black, color=black, inner sep=0.5pt, minimum width=5pt, label = 
    $g'$](g){};
    \draw (2,0)  node [draw=none,fill=black, color=black, inner sep=0.5pt, minimum width=5pt, label=$g'h$]{};
    \draw (7,-1)  node [draw=none,fill=black, color=black, inner sep=0.5pt, minimum width=5pt, label = {0:$g'w$}](gw){};
    \draw (5.5,0)  node [draw=none,fill=none, color=black, inner sep=0.5pt, minimum width=5pt, label = $w$](w){};
    \draw (3.5,-1)  node [draw=none,fill=none, color=black, inner sep=0.5pt, minimum width=5pt, label ={-90:$u$}](u){};
    \draw (1,0)  node [draw=none,fill=none, color=black, inner sep=0.5pt, minimum width=5pt, label ={-90:$v_{g'h}$}](){};
    \end{tikzpicture}
    \caption{Case 2.2.A: we factorise $v=v_{g'h}u$.}
    \label{fig:22a}
    \end{subfigure}%
    \vspace{1.5em}
    \begin{subfigure}[t]{\textwidth}
    		\centering
    \begin{tikzpicture}[thick,scale=.6,arrowmark/.style 2 args={decoration={markings,mark=at position #1 with \arrow{#2}}}]%
      \tikzstyle{every node}=[circle, draw, fill=blue, color=blue,
                            inner sep=0pt, minimum width=6pt]
     \begin{pgfonlayer}{background}                       
                            \draw[-,>=stealth] (0,0) -- (2,0) -- (4, 0);
                            \draw[-] (4, 0) to[out=00, in=100, looseness=0.8] (7,-1);
                            \draw[-] (4,-2) to[out=0, in=-100, looseness=1]  (7, -1);
                            \draw[-] (2,0) to[out=-100, in=180, looseness=0.9]  (4, -2);
    \end{pgfonlayer}                      
    \draw (0,0)  node[draw=none,fill=black, color=black, inner sep=0.5pt, minimum width=5pt, label = $1_G$] {};
    \draw (4,0)  node [draw=none,fill=black, color=black, inner sep=0.5pt, minimum width=5pt, label = 
    $g'$](g){};
    \draw (2,0)  node [draw=none,fill=black, color=black, inner sep=0.5pt, minimum width=5pt]{};
    \draw (4,-2)  node [draw=none,fill=black, color=black, inner sep=0.5pt, minimum width=5pt, label = 
    {-90:$g'h$}](gh){};
    \draw (7,-1)  node [draw=none,fill=black, color=black, inner sep=0.5pt, minimum width=5pt, label = {0:$g'w$}](gw){};
    \draw (5.5,0)  node [draw=none,fill=none, color=black, inner sep=0.5pt, minimum width=5pt, label = $w$](w){};
    \draw (5.5,-2)  node [draw=none,fill=none, color=black, inner sep=0.5pt, minimum width=5pt, label ={-90:$u$}](u){};
    \draw (1,0)  node [draw=none,fill=none, color=black, inner sep=0.5pt, minimum width=5pt, label = 
    {$v_1$}](v1){};
    \draw (3,0)  node [draw=none,fill=none, color=black, inner sep=0.5pt, minimum width=5pt, label = 
    {$w_1$}](w1){};
    \draw (2.5,-1.5)  node [draw=none,fill=none, color=black, inner sep=0.5pt, minimum width=5pt, label = 
    {-90:$v_2$}](v2){};
    \end{tikzpicture}
    \caption{Case 2.2.B.I: here $\slrep(g')w=v_1w_1w$, and $v=v_{g'h}u=v_1v_2u$.}
    \label{fig:22bi}
    \end{subfigure}%
    \vspace{1.5em}
    \begin{subfigure}[t]{\textwidth}
    \centering
    \begin{tikzpicture}[thick,scale=.6,arrowmark/.style 2 args={decoration={markings,mark=at position #1 with \arrow{#2}}}]%
      \tikzstyle{every node}=[circle, draw, fill=blue, color=blue,
                            inner sep=0pt, minimum width=6pt]
     \begin{pgfonlayer}{background}                       
                            \draw[-,>=stealth] (0,0) -- (2,0) -- (4, 0);
                            \draw[-] (4, 0) to[out=00, in=100, looseness=0.8] (7,-1);
                            \draw[-] (4,-2) to[out=0, in=-100, looseness=1]  (7, -1);
                            \draw[-] (2,0) to[out=-100, in=180, looseness=1.3]  (4, -2);
    \end{pgfonlayer}                      
    \draw (0,0)  node[draw=none,fill=black, color=black, inner sep=0.5pt, minimum width=5pt, label = $1_G$] {};
    \draw (4,0)  node [draw=none,fill=black, color=black, inner sep=0.5pt, minimum width=5pt, label = 
    $g'$](g){};
    \draw (4,-2)  node [draw=none,fill=black, color=black, inner sep=0.5pt, minimum width=5pt, label = 
    {-90:$g'h$}](gh){};
    \draw (7,-1)  node [draw=none,fill=black, color=black, inner sep=0.5pt, minimum width=5pt, label = {0:$g'w$}](gw){};
    \draw (5.5,0)  node [draw=none,fill=none, color=black, inner sep=0.5pt, minimum width=5pt, label = $w$](w){};
    \draw (5.5,-2)  node [draw=none,fill=none, color=black, inner sep=0.5pt, minimum width=5pt, label ={-90:$u$}](u){};
    \end{tikzpicture}
    \caption{Case 2.2.B.II.}
    \label{fig:22bii}
    \end{subfigure}
    \caption{Some of the cases studied in \cref{lem: fftp implies finite number of sl cones}: straight lines correspond to $\slg$ representatives, the upper path from $1_G$ to $g'w$ is given by $\slrep(g')w$, and the lower path with the same endpoints is given by $v$ and has $u$ as a suffix.}
    \label{fig:cones}
\end{figure}

\textbf{Case 1.} Suppose $|v|<|\slrep(g')w|$. By SL-FFTP, there exists $h\in B_K(1_G)$ such that the vertex $g'h$ lies in the path labelled by $v$. Then, if $u$ denotes the suffix of $v$ starting at $g'h$, we have the following:
 \[\begin{array}{rcl}
        |\slrep(gh)u| & = & |\slrep(gh)|+|u|-|\slrep(g)|+|\slrep(g)| \\
         & \stackrel{\Delta_g(h)=\Delta_{g'}(h)}{=} & |\slrep(g'h)|+|u|-|\slrep(g')|+|\slrep(g)| \\
         & \leq & |v|-|\slrep(g')|+|\slrep(g)| \\
         & < & |\slrep(g')|+|w|-|\slrep(g')|+|\slrep(g)| = |\slrep(g)w|,
         \end{array}\]
         which contradicts the fact that $\slrep(g)w\in \slg$.

\textbf{Case 2.} Assume $|v|=|\slrep(g')w|$ (in particular, $v\lex \slrep(g')w$). 

\textbf{Case 2.1} If $\slrep(g')$ is a prefix of $v$, we denote by $u$ the suffix of $v$ starting at $g'$. 
So, in this case we have that $v= \slrep(g')u$, $v=_G \slrep(g')w$ and $|v|=|\slrep(g')w|$.
Then $u=_G w$, $u\lex w$ and $|u|=|w|$. From this last equality, $|\slrep(g)u|=|\slrep(g)w|$. Therefore, $\slrep(g)u\lex\slrep(g)w$, which contradicts the fact that $w\in \csl(g)$.

\textbf{Case 2.2} Assume that $\slrep(g')$ is not a prefix of $v$.
By $\slfftp$, there exists $h\in B_K(1_G)$ so that $g'h$ is a vertex lying in the  path described by $v$ starting at $1_G$. 
There are two cases depending on whether $g'h$ is a vertex in the path  described by  $\slrep(g')$ starting at $1_G$ or not.

\textbf{Case 2.2.A} Suppose that  $g'h$ is a vertex lying in the path  described by  $\slrep(g')$ starting at $1_G$.  We denote by $v_{g'h}$ the prefix of $v$ that labels the path from $1_G$ to the element $g'h$. And we denote by $u$ the suffix of $v$ starting at $g'h$. Observe that $v=v_{g'h}u$. Now, because $g'h$ lies in the path described by $\slrep(g')$, we have that $\slrep(g'h)$ is a prefix of $\slrep(g')$ and, as a consequence, $\slrep(g'h)u$ and $\slrep(g')w$ $K$-fellow travel (notice that they also represent the same element of the group). By definition, $\slrep(g'h)\leq_{\text{SL}} v_{g'h}$, so $\slrep(g'h)u \leq_{\text{SL}} v_{g'h}u=v\sll \slrep(g')w$. Since the three main properties of the path $v$ given by SL-FFTP are also satisfied by  $\slrep(g'h)u$, we can take $v=\slrep(g'h)u$ without loss of generality.

We know that $v\sll \slrep(g')w$ and $|v|=|\slrep(g')w|$, so $v\lex \slrep(g')w$. We factorise $\slrep(g')w=\slrep(g'h)w_1w$ and recall that $v=\slrep(g'h)u$.  As $\slrep(g'h)$ is a prefix of both $v$ and $\slrep(g')$, we have that $u\lex w_1w$ and, because $|u|=|w_1w|$, we have $u\sll w_1w$. Therefore, $\slrep(gh)u\sll \slrep(gh)w_1w=\slrep(g)w$, where the last word equality is true because $w_1=_G h^{-1}$ and $w_1\in \slg$ since it is a subword of the shortlex word $\slrep(g')$. 

\textbf{Case 2.2.B} Suppose that  $g'h$ is not  a vertex lying in the paths described by $v$ and $\slrep(g')$ starting at $1_G$. 
We will denote by $v_{g'h}$ the prefix of $v$ that labels the path from $1_G$ to the element $g'h$. And we denote by $u$ the suffix of $v$ starting at $g'h$.

\textbf{Case 2.2.B.I} Suppose $|v_{g'h}|=|\slrep(g'h)|$. 

    Let us write $v_{g'h}=v_1v_2$, where $v_1$ is a maximal prefix of $\slrep(g')$ and also factorise $\slrep(g')w=v_1w_1w$. By the maximality in the choice of $v_1$ and since $\slrep(g')$ is not a prefix of $v$, we have $v_2\lex w_1$ and, therefore, $v_{g'h}\lex \slrep(g')$. Also, observe that because $\slrep(g'h)\leq_{\text{SL}}v_{g'h}$ and we are assuming that $|v_{g'h}|=|\slrep(g'h)|$, we have that $\slrep(g'h)\leq_{\text{lex}} v_{g'h}$. By transitivity, $\slrep(g'h)\lex \slrep(g')$. Because we are in case 2.2, $\slrep(g'h)$ is not a prefix of $\slrep(g')$. Now, since $C_g=C_{g'}$, and using that $P_g(h)=P_{g'}(h)$ and $L_g(h)=L_{g'}(h)$,  we also have that  $\slrep(gh)\lex \slrep(g)$ and $\slrep(gh)$ is not a prefix of $\slrep(g)$.

    We have:
 \[\begin{array}{rcl}
        |\slrep(gh)u| & = & |\slrep(gh)|+|u|-|\slrep(g)|+|\slrep(g)| \\
         & \stackrel{\Delta_g(h)=\Delta_{g'}(h)}{=} & |\slrep(g'h)|+|u|-|\slrep(g')|+|\slrep(g)| \\
         & = & |v|-|\slrep(g')|+|\slrep(g)| \\
         & = & |\slrep(g')|+|w|-|\slrep(g')|+|\slrep(g)| = |\slrep(g)w|.
         \end{array}\]
    From this equality, together with the two previous observations, we conclude that $\slrep(gh)u\sll \slrep(g)w$, contradicting the fact that $w\in \csl(g)$.
    
\textbf{Case 2.2.B.II} Suppose $|v_{g'h}|>|\slrep(g'h)|$. Then we have:
 \[\begin{array}{rcl}
        |\slrep(gh)u| & = & |\slrep(gh)|+|u|-|\slrep(g)|+|\slrep(g)| \\
         & \stackrel{\Delta_g(h)=\Delta_{g'}(h)}{=} & |\slrep(g'h)|+|u|-|\slrep(g')|+|\slrep(g)| \\
         & < & |v_{g'h}|+|u|-|\slrep(g')|+|\slrep(g)| \\
         & = &  |v|-|\slrep(g')|+|\slrep(g)| \\
         & = & |\slrep(g')|+|w| -|\slrep(g')|+|\slrep(g)| \\
         & = & |\slrep(g)w|.
         \end{array}\]
This contradicts the fact that $w\in \csl(g)$.
\end{proof}

\begin{corollary} \label{prop: fftp implies regular language}
    Let $G$ be a group with finite generating set $\Sigma$ satisfying $\slfftp$, then $\slg$ is a regular language.
\end{corollary}
The following corollary can be proved as done in \cite[Proposition 4.2]{Neumann_Shapiro_1995} for the case of geodesics, in fact, the case of shortlex representatives is simpler since the $\slg$ is in bijection with the group.
\begin{corollary}\label{coro:slfftp_rational_growth}
    Let $G=\langle \Sigma\rangle$ be a group with regular set of shortlex representatives $\slg$ for some ordering of $\Sigma\cup\Sigma^{-1}$. Then, $G$ has rational growth function with respect to this generating set. 
\end{corollary}
As it happens in the analogous case of geodesics, the converse to the previous result does not hold:
    The Heisenberg group $G = \langle a, b \;|\; [a, b]\text{ is central } \rangle $ has rational growth with respect to the generating set $\{a, b\}$ but the set of shortlex representatives is not regular with respect to any ordering of the generating set.
The fact that the group has rational growth follows from the main result in \cite{duchinshapiro19}. The arguments in \cite[Proposition 1.1 and Proposition 1.5]{warshall10} that show that the set of geodesics in $G$ is not regular works exactly the same to see that $\slg$ is not regular either.

\section{SL-cones on an example of Cannon}\label{sec:fftp_vs_slfftp}

In \cite{elder05},  Elder furnishes an example, that he attributes to Cannon, of a group and a generating set whose set of geodesics is regular and does not satisfy the falsification by fellow traveller property. The group in question is given by 
$$G= \langle a, t \;|\; t^2=1, atat =tata\rangle.$$
We will use this example to illustrate the relation between the falsification by fellow traveller property and its shortlex version. More specifically, we will see that the set of shortlex representatives $\slg$ is a regular language for any total ordering of the generators $\{a, a^{-1}, t\}$ but that there is an ordering of $\{a, a^{-1}, t\}$ such that the group does not have the shortlex FFTP. Moreover, there are orderings for which the group does have the shortlex FFTP.

Before we go into the proofs, we refer the reader to \cite{elder05} for further details and intuition about the elements of the group $G$, its Cayley graph and the paths in it. In particular, we will use the fact that, having chosen a ``base vertex'' in the Cayley graph (representing $1_G$), each vertex in the Cayley graph and hence each element in the group can be uniquely specified by a triple $(x, y, \varepsilon)$ where $x,y\in\Z$ and $\varepsilon\in\{+, -\}$. Intuitively, this triple comes from the fact that the vertices of the Cayley graph of $G$ with respect to the given generators are in a natural bijection with $\Z^2\times \{+,-\}$: we refer to $\Z^2\times\{-\}$ as the \emph{bottom layer} and to $\Z^2\times\{+\}$ as the \emph{top layer} of the graph. Multiplication by the generator $t$ maps a vertex in one of the layers to the same vertex in the other layer. Multiplication by the generator $a$ determines horizontal displacement in the bottom layer, and determines vertical displacement in the top layer. We set the identity $1_G$ to have coordinates $(0,0,-)$.

Lastly, a word on the geodesics in $G$ will be needed for the following proofs. We recall from \cite[Theorem 3.5]{elder05} that the set 
$$\{a^x,\; a^xta^y,\; a^{x_1}ta^{y'}ta^{x_2}:\;x, x_1, x_2, y \in \Z, x_1\cdot x_2\geq 0, y'\in \Z-\{0\} \}$$
is the language of all geodesics in $G$ with the given generating set. Moreover, following \cite[Lemma 3.1]{elder05} we know that the subset of words of the form $a^xta^y$ are \emph{unique} geodesics, that is, for each of those words there is no other geodesic representing the same group element.  The words of the form $a^x$ are easily seen to be unique geodesics as well by taking into account that the expression of the corresponding Cayley graph's vertex with the coordinates described above is $(x, 0, -)$. Note that, because of the uniqueness, these two families of geodesics are indeed shortlex representatives in $\slg$. 

The following lemma will be used to manipulate words in the rest of the section:
\begin{lemma}\label{lemma:exponent_game}
    Let $G=\langle a, t\rangle$ as above. For any integers $x_1, x_2, y_1, y_2$ we have:
    $$a^{x_1}ta^{y_1}ta^{x_2}ta^{y_2}=_G a^{x'_1}ta^{y_1'}ta^{x'_2}ta^{y_2'} $$ if and only if $x_1+x_2=x_1'+x_2', \; y_1+y_2=y_1'+y_2'.$ In particular, 
    $$a^{x_1}ta^{y}ta^{x_2} =_G a^{x'_1}ta^{y}ta^{x'_2}$$
    if and only if $x_1+x_2=x_1'+x_2'$.
\end{lemma}
\begin{proof}
    Suppose $x_1+x_2=x_1'+x_2', \; y_1+y_2=y_1'+y_2'.$ Since $a(tat)=(tat)a$, we also have that $a(tat)^n=(tat)^na$ and, since $t^2=1$, that $a(ta^nt)=(ta^nt)a$. Using this equality, we see that $a^{x_1}ta^{y_1}ta^{x_2}ta^{y_2}=_Ga^{x_1+x_2}ta^{y_1+y_2}$ and $a^{x_1'}ta^{y_1'}ta^{x_2'}ta^{y_2'}=_Ga^{x_1'+x_2'}ta^{y_1'+y_2'}$, so this implication is verified.

    Suppose now that $a^{x_1}ta^{y_1}ta^{x_2}ta^{y_2}=_G a^{x'_1}ta^{y_1'}ta^{x'_2}ta^{y_2'} $. A word of the form $a^{x_1}ta^{y_1}ta^{x_2}ta^{y_2}$ represents the vertex with coordinates $(x_1+x_2,y_1+y_2,+)$ in the Cayley graph, therefore $(x_1+x_2,y_1+y_2)=(x_1'+x_2',y_1'+y_2')$ and we obtain the other implication.
\end{proof}
\begin{proposition}\label{prop:ordering_non_slfftp}
    The group $G$ does not satisfy the shortlex FFTP for the shortlex ordering induced by $a<t<a^{-1}$.
\end{proposition}
\begin{proof}
    Let $w = ta^yta^xt$, with $y>0$ and $x<0$.
    Because $a^xta^y$ represents the same group element as $w$, $w\not\in\slg$. We will see that any word $w'\sll w$ with $w'$ representing the same group element as $w$ remains ``far'' from $w$ in the Cayley graph, and therefore choosing large enough $|x|$, $|y|$ would contradict the existence of a uniform fellow travelling constant.

    The word $w= ta^yta^xt$ corresponds to an element $g\in G$ whose coordinates in the Cayley graph are $(x,y,+)$, and this means that any word representing  $g$ has, at least, $|x|+|y|$ occurrences of the letters $a$, $a^{-1}$. From \cite[Lemma 3.1]{elder05} we know that the parity of occurrences of $t$ is preserved among representatives of the same element, so a word $w'\sll w$ with $w'$ representing $g$ must be one of the following:
    \begin{itemize}
        \item A word of length $|w'|=|x|+|y|+1<|w|$ with exactly one occurrence of $t$.
        \item A word of length $|w'|=|w|=|x|+|y|+3$ with exactly three occurrences of $t$ and $|w'|\lex|w|$.
    \end{itemize}
    If $|w'|=|x|+|y|+1$, $w'$ must be $a^xta^y$ because this word is a unique geodesic as described above. In this case, $w$ and $w'$ have vertices at distance $\sim |x|+|y|$, which can be made arbitrarily large by changing $w$, hence contradicting the existence of a uniform fellow travelling constant $K$. 

    Therefore, a shortlex fellow-travelling word $w'$ for $w$ must be of the second type, with $|w'|=|w|$ and $w'\lex w$. In fact, $$w' = a^{x_1}ta^{y_1}ta^{x_2}ta^{y_2}.$$
    From \cref{lemma:exponent_game}, $x_1+x_2=x$ and $y_1+y_2=y$. Moreover,  $x_1, x_2\leq 0$ and $y_1,y_2\geq 0$ because $|w|=|w'|$. It is then clear that if $x_1\neq 0$, $w\lex w'$ because $t < a^{-1}$, and therefore we can assume that $w' = ta^{y_1}ta^{x_2}ta^{y_2}.$ Now, if $y_2\neq 0$ we must have $y_1<y$ and therefore $w\lex w'$ because of the prefixes $ta^{y_1+1}\lex ta^{y_1}t$ and $a < t$. We conclude that $y_2=0$, so $w'=w$ and the only word $w'\sll w$ is  the geodesic $a^xta^y$ discussed above. It follows that $G$ does not satisfy the shortlex FFTP for this ordering of the generators.
\end{proof}
\begin{remark}\label{remark:another_order}
    Note that a symmetric counterexample could have been given for the ordering $a^{-1}<t<a$.
\end{remark}
 We next show that the language $\slg$ is regular for any choice of ordering of the generators.
\begin{theorem}\label{prop:slg_regular_not_slfftp}
    Let $$G= \langle a, t \;|\; t^2=1, atat =tata\rangle.$$
    For any fixed total order on the set $\{a, a^{-1}, t\}$, the language of shortlex representatives $\slg$ is regular and, therefore, $G$ has finitely many $\slrep$-cones with respect to this generating set.
\end{theorem}

\begin{proof}
We will see that any total order on $\{a, a^{-1}, t\}$ yields a regular language $\slg$.

Following the description of geodesics given at the beginning of the section, and recalling that the words of types $a^x$ and $a^xta^y$ are unique geodesics, we know that the shortlex representative of a geodesic of the form $a^{x_1}ta^yta^{x_2}$ with $x_1,x_2\in \Z$, $y\in\Z-\{0\}$ and $x_1x_2\geq 0$ is another word in the same family. 

\cref{lemma:exponent_game} makes it immediate to decide what the shortlex representative of a word $a^{x_1}ta^yta^{x_2}$ is 
depending on the order that we fix on the set $\{a,a^{-1},t\}$. Note that we only consider $y\in \Z\setminus\{0\}$ because the case $y = 0$ yields the word $a^{x_1+x_2}$ as its unique geodesic and therefore as its shortlex representative. Now:
\begin{itemize}
    \item If $a<a^{-1}<t$ or $a^{-1}<a<t$, for all $y\in \mathbb{Z}\backslash\{0\}$ we have $$\slrep(a^{x_1}ta^yta^{x_2}) = a^{x_1+x_2}ta^{y}t.$$ 
    \item If $t<a<a^{-1}$ or $t<a^{-1}<a$,  for all $y\in \mathbb{Z}\backslash\{0\}$ we have $$\slrep(a^{x_1}ta^yta^{x_2}) = ta^{y}ta^{x_1+x_2}.$$
    \item If $a<t<a^{-1}$, for all $x_1+x_2\geq 0$, $y\neq 0$ we have $$\slrep(a^{x_1}ta^yta^{x_2}) = a^{x_1+x_2}ta^{y}t$$   and  for all for all $x_1+x_2\leq 0$, $y\neq 0$ we have $$\slrep(a^{x_1}ta^yta^{x_2}) = ta^{y}ta^{x_1+x_2}.$$     
    \item If $a^{-1}<t<a$,  for all $x_1+x_2\leq 0$, $y\neq 0$  we have that $$\slrep(a^{x_1}ta^yta^{x_2})  = a^{x_1+x_2}ta^{y}t$$ and  for all $x_1+x_2\geq 0$, $y\neq 0$, we have that $$\slrep(a^{x_1}ta^yta^{x_2})  = ta^{y}ta^{x_1+x_2}.$$ 
\end{itemize}
For each of the possible total orders on $\{a,a^{-1},t\}$, it is clear from the description above that the language of shortlex representatives of the words $$\{a^{x_1}ta^yta^{x_2}:\; x_1x_2\geq 0, y\neq 0\}$$
is a regular language. As we previously observed, the shortlex representatives of the remaining elements in the group are given by words in $$\{a^x,\; a^xta^y:\;x, y \in \Z\},$$
which is the union of two regular languages. Therefore, given any total order on the generating set $\{a,a^{-1},t\}$, the language $\slg$ of shortlex representatives of the group $G$ is given by the union of three regular languages which is again a regular language.
\end{proof}
\cref{fig:autAat,fig:auttAa,fig:autAta,fig:autatA} show the finite state automaton given by the construction in \cref{lemma: construction automaton} in every case. In each automaton, all states are accepting and the initial state is represented by the fully colored circle in black. The circles depicting the states are decorated so that they are reminiscent of the shortlex cone they represent. Following the terminology introduced in \cite{elder05}, black states correspond to SL-cones of elements in the bottom layer of the a Cayley graph, while the red ones correspond to elements in the top layer. The filling pattern of each state corresponds to the horizontal and vertical coordinates of the elements whose SL-cone they represent. 
\begin{figure}[h]
    \centering
    \begin{subfigure}[t]{0.5\textwidth}
        \centering
		\includegraphics[width=\linewidth]{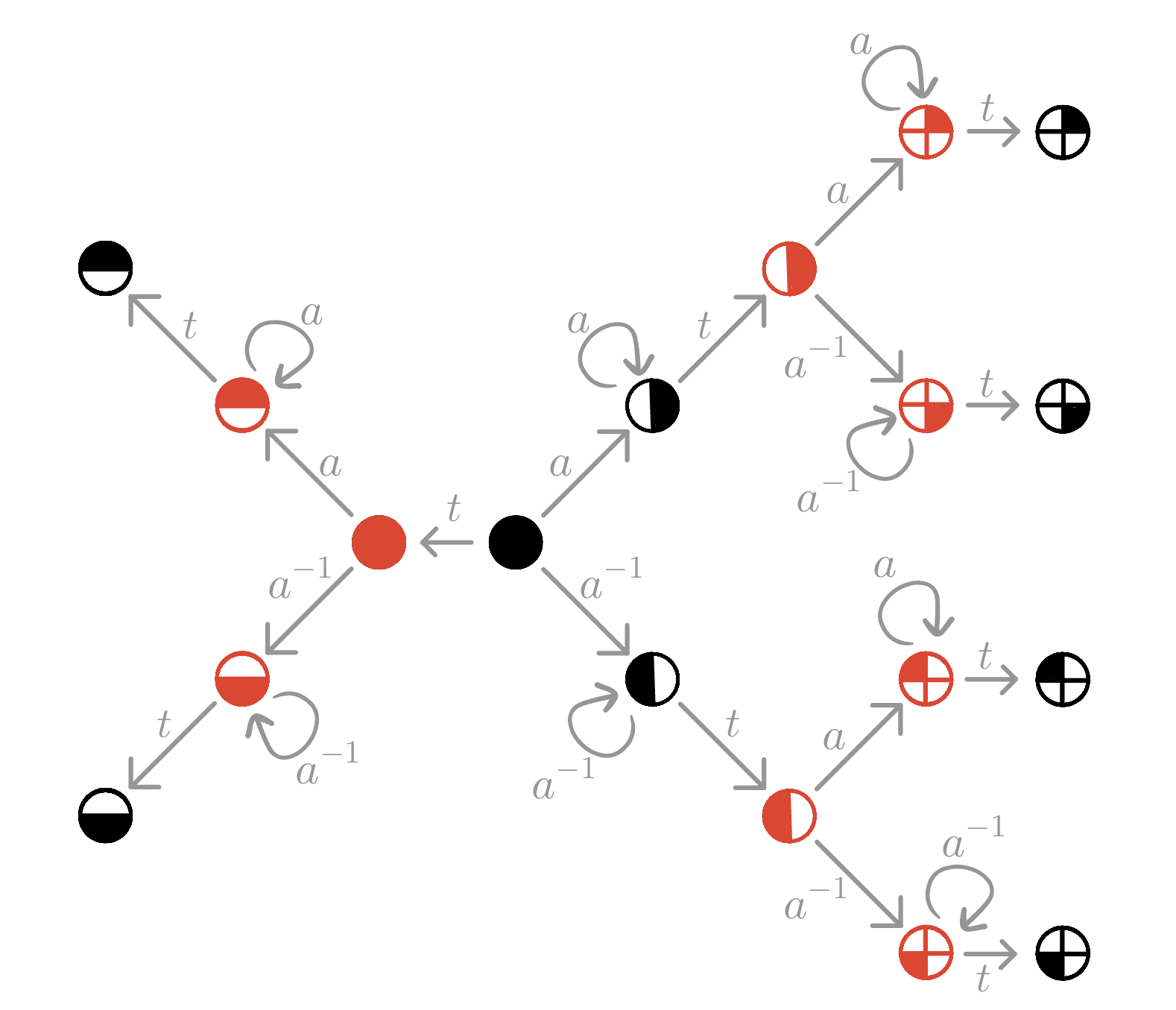}
		\caption{ Automaton for $\slg$ when $a<a^{-1}<t$.}
		\label{fig:autAat}
    \end{subfigure}%
    ~ 
    \begin{subfigure}[t]{0.5\textwidth}
        \centering
		\includegraphics[width=\linewidth]{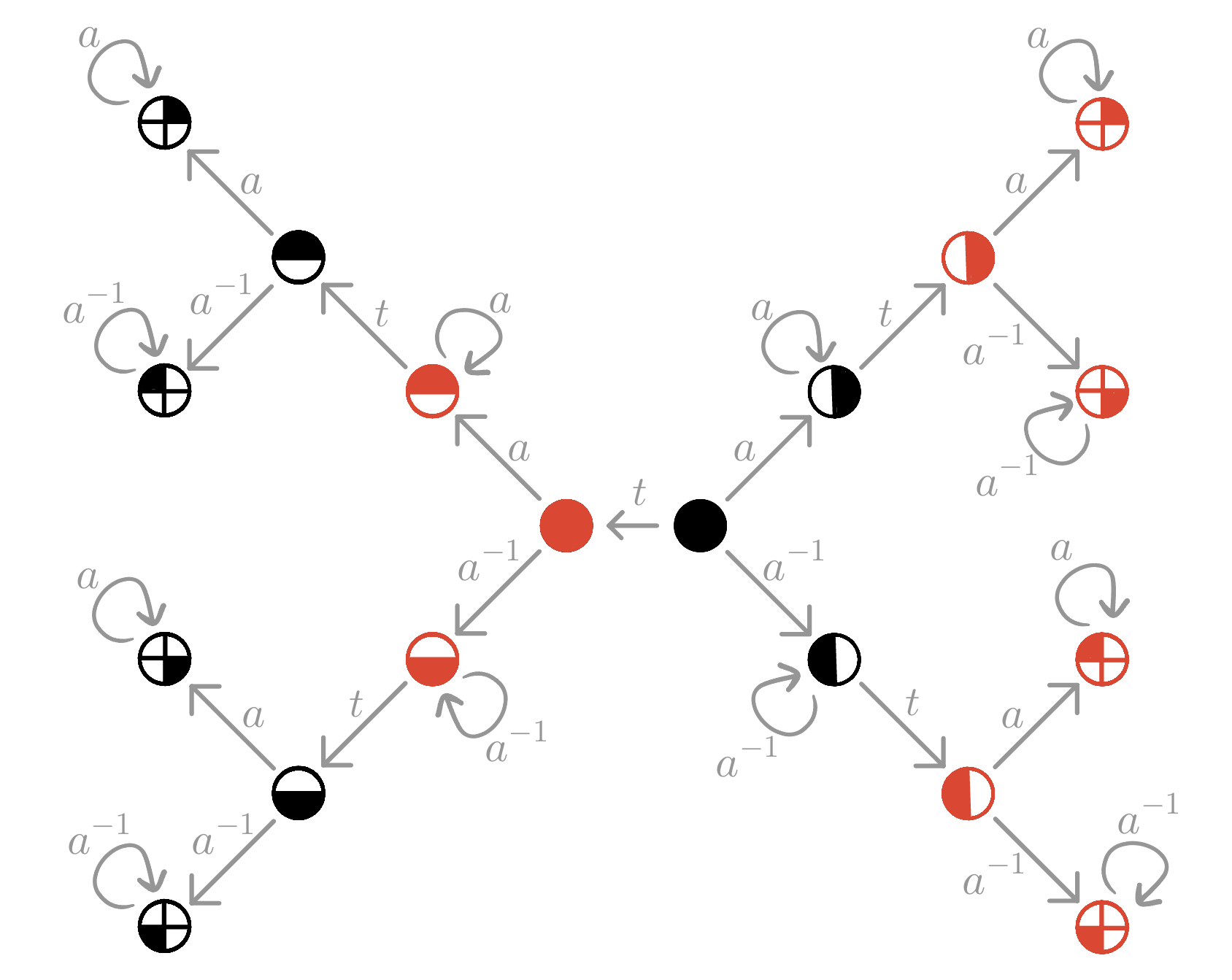}
		\caption{Automaton for $\slg$ when $t<a<a^{-1}$.}
		\label{fig:auttAa}
    \end{subfigure}

     \begin{subfigure}[t]{0.5\textwidth}
        \centering
		\includegraphics[width=\linewidth]{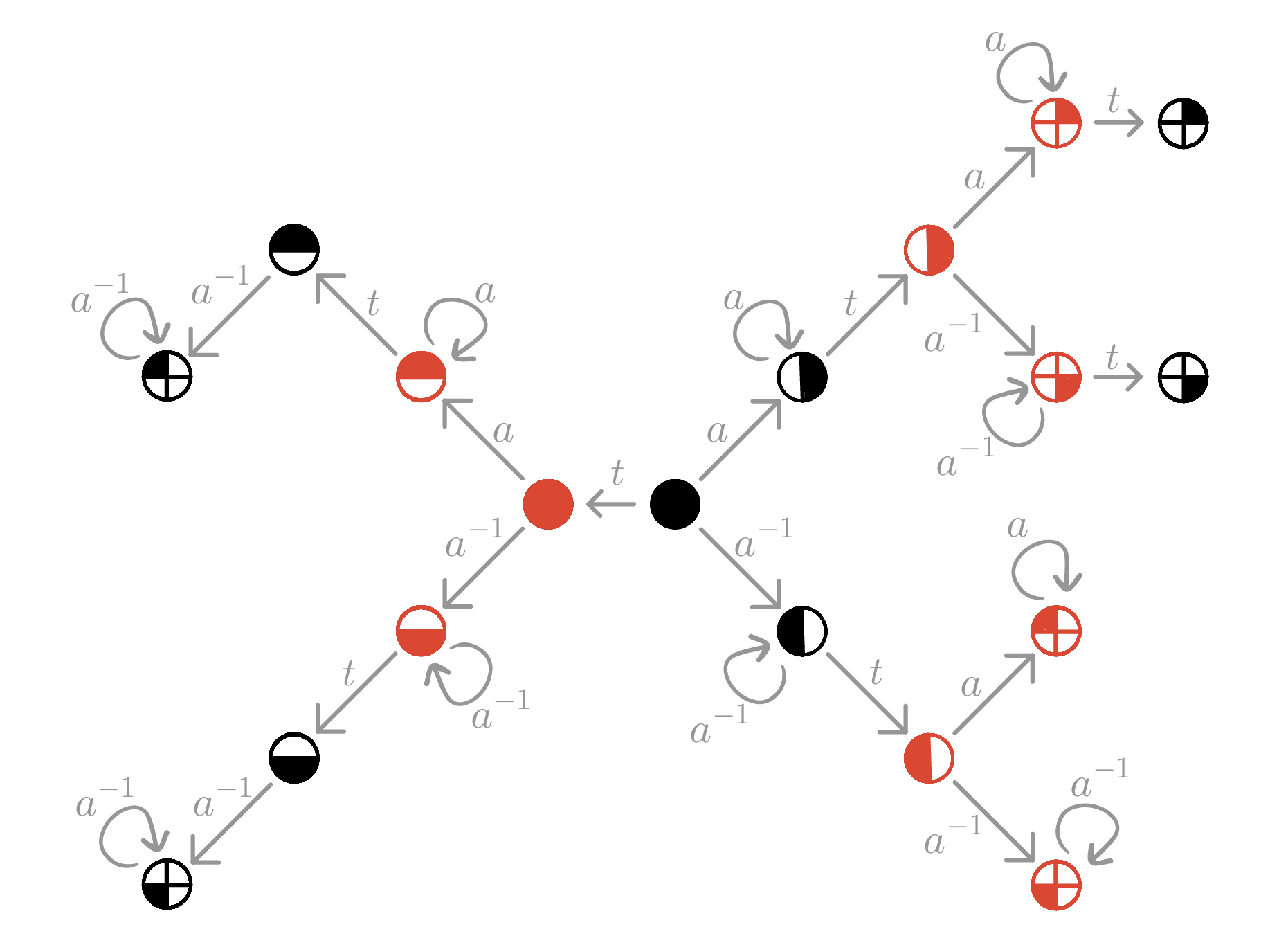}
		\caption{Automaton for $\slg$ when $a<t<a^{-1}$.}
		\label{fig:autAta}
    \end{subfigure}%
    ~ 
    \begin{subfigure}[t]{0.5\textwidth}
        \centering
		\includegraphics[width=\linewidth]{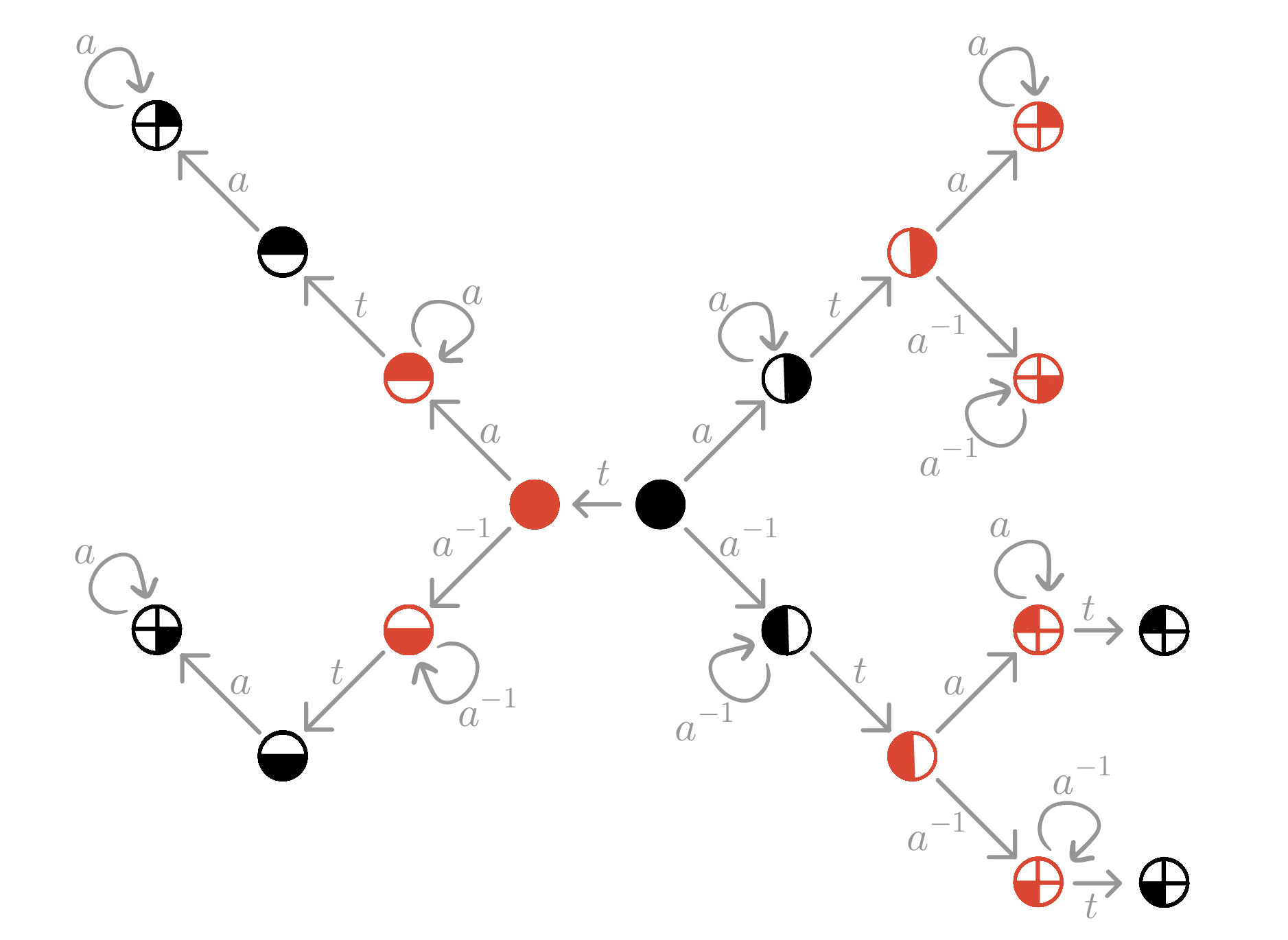}
		\caption{Automaton for $\slg$ when $a^{-1}<t<a$.}
		\label{fig:autatA}
    \end{subfigure}
\end{figure}

To conclude the comparison between FFTP and shortlex FFTP, we prove that for the orderings that were not considered in \cref{prop:ordering_non_slfftp,remark:another_order}, the group $G$ exhibits the shortlex FFTP. 
\begin{proposition}\label{prop:some_orderings_slfftp}
    For  any of the orderings $a<a^{-1}<t$, $a^{-1}<a<t$, $t<a<a^{-1}$, $t<a^{-1}<a$, the group $G$ has the shortlex FFTP.
\end{proposition}
\begin{proof}
    The case $a<a^{-1}<t$ is identical to the case $a^{-1}<a<t$, and the case $t<a<a^{-1}$ is identical to $t<a^{-1}<a$. The case $t<a<a^{-1}$ is symmetrical to the case $a<a^{-1}<t$, so we reduce the problem to the ordering $a<a^{-1}<t$.

    Let $w$ be a word over the generators so that $w\not \in \slg$, let us see that there is a global constant $K$ so that we can find a word $w'\sll w$ representing the same group element and $K$-fellow travelling with $w$. 
    
    We distinguish two cases: $w$ is a geodesic and therefore has the form $a^x$, $a^xta^y$ or $a^{x_1}ta^yta^{x_2}$ with $x_1x_2\geq 0$, or else $w$ is not geodesic and contains a subword $ta^yta^xt$. The first two families are unique geodesics as observed at the beginning of the section and therefore are already in $\slg$, so we only need to consider the remaining two families.
    
    For the words $a^{x_1}ta^yta^{x_2}$ it is convenient to separate the case where $x_1=0$ and $ x_2\neq 0$, the case $x_1\neq 0$ and $x_2=0$, and the case $x_1x_2>0$. In all three cases we can assume $y\neq 0$. For the subwords $ta^yta^xt$, we can assume that $x\neq 0$ and $y\neq 0$. We have: 
    \begin{enumerate}
        \item If $w=ta^yta^{x_2}$, the word $w'=a^{\pm 1}ta^yta^{x_2\mp1}$ (the choice of $\pm$ depends on the sign of $x_2$) fellow travels with $w$ at distance at most $3$, and satisfies $w'\sll w$.
        \item If $w=a^{x_1}ta^yt$, $w$ is already a shortlex representative.
        \item If $w=a^{x_1}ta^yta^{x_2}$ with $x_1x_2>0$, the word $w'=a^{x_1\pm 1}ta^yta^{x_2\mp1}$ (the choice of $\pm$ depends on the sign of $x_1$) fellow travels with $w$ at distance at most $3$, and satisfies $w'\sll w$.
        \item If $w$ contains a subword $v = ta^yta^xt$, $v$ can be replaced by the word $v'=a^{\pm 1}ta^yta^{x\mp 1}t$ (the choice of $\pm$ depends on the sign of $x$) which fellow travels with $v$ at distance at most $3$ and satisfies $v'\sll v$.
    \end{enumerate}
    
\end{proof}

\noindent {\bf Acknowledgements}. The authors would like to thank Yago Antolín, advisor of the second author, for all the help during the process of carrying out this project and for his careful reading of the drafts of this work. Paloma López Larios acknowledges support by the grant PRE2022-101796, financed by MCIN/AEI/10.13039/ 501100011033 and by FSE+, and also by the Severo Ochoa Grant CEX2023-001347-S funded by MICIU/AEI/10.13039/501100011033. Lucía Asencio Martín acknowledges support by Leverhulme foundation grant no.RPG-2022-025 during the work on this article. Both authors acknowledge support by the projects PID2021-126254NB-I00 and PID2024 155800NB-C32.

\bibliographystyle{abbrv}
\bibliography{referencias}

\end{document}